\documentclass[a4paper,reqno,11pt]{amsart}
\usepackage{amsmath}
\usepackage{amssymb}
\usepackage{cases}
\usepackage{enumitem}
\allowdisplaybreaks[4]

\newtheorem{thm}{Theorem}[section]
\newtheorem{prop}[thm]{Proposition}

\newtheorem{lem}[thm]{Lemma}

\newtheorem{cor}[thm]{Corollary}

\newtheorem{claim}[thm]{Claim}

\theoremstyle{definition}
\newtheorem{definition}[thm]{Definition}

\theoremstyle{remark}

\numberwithin{equation}{section}

\newcommand{\bQ}{\mathbb{Q}}
\newcommand{\bR}{\mathbb{R}}

\newcommand{\bP}{\mathbb{P}}

\newcommand\OO{{\mathcal{O}}}

\newcommand{\mult}{\operatorname{mult}}
\newcommand{\Nklt}{\operatorname{Nklt}}

\newcommand{\Supp}{\operatorname{Supp}}
\newcommand{\lct}{\operatorname{lct}}

\newcommand{\Vol}{\operatorname{Vol}}
\newcommand{\length}{\operatorname{length}}

\usepackage{todonotes}

\begin{document}

\title{An effective upper bound for anti-canonical volumes of canonical $\mathbb{Q}$-Fano threefolds}
\date{\today}
\author{Chen Jiang}
\address{Chen Jiang, Shanghai Center for Mathematical Sciences, Fudan University, Jiangwan Campus, Shanghai, 200438, China}
\email{chenjiang@fudan.edu.cn}

\author{Yu Zou}
\address{Yu Zou, School of Mathematical Sciences, Fudan University, Shanghai, 200433, China}
\email{19110180016@fudan.edu.cn}


\begin{abstract}
We show that the anti-canonical volume of a canonical weak $\bQ$-Fano $3$-fold is at most $324$.
\end{abstract}

\keywords{weak $\bQ$-Fano threefolds, anti-canonical volumes, log canonical thresholds, boundedness}
\subjclass[2010]{14J45, 14J30, 14J17}
\maketitle
\pagestyle{myheadings} \markboth{\hfill C.~Jiang \& Y.~Zou
\hfill}{\hfill Effective bound for anti-canonical volumes of canonical $\bQ$-Fano $3$-folds\hfill}

\tableofcontents

\section{Introduction}
Throughout this paper, we work over the field of complex numbers $\mathbb{C}$.

A normal projective variety $X$ is called a {\it
weak $\bQ$-Fano variety} (resp. {\it $\bQ$-Fano variety}) if the anti-canonical divisor $-K_X$ is nef and big (resp. ample). 
A {\it canonical} (resp. {\it terminal}) weak $\bQ$-Fano variety is a weak $\bQ$-Fano variety with at worst canonical (resp. terminal) singularities. 

According to the minimal model program, weak $\bQ$-Fano varieties form a fundamental class among research objects of birational geometry. Motivated by the classification theory of $3$-dimensional algebraic varieties, we are interested in the study of explicit geometry of canonical or terminal (weak) $\bQ$-Fano $3$-folds. In this direction, there are a lot of works in the literature (for instance, \cite{S04, Prok05, BS07a, BS07b, Prok07, CC08, Prok10, Chen11, Prok13, PR16, CJ16, CJ20}). 

By \cite{Kaw92, KMMT00}, canonical weak $\bQ$-Fano $3$-folds form a bounded family (see \cite{Bir16b} for the much deeper higher dimensional generalization), so it is interesting to ask for effective or optimal bounds of different kinds of invariants of canonical weak $\bQ$-Fano $3$-folds. See \cite[Section~2.4]{CJ20} for a summary of known results.

Given a canonical weak $\bQ$-Fano $3$-fold $X$, we are interested in the {\it anti-canonical volume} $(-K_X)^3$ of $X$. 
This is an important invariant of $\bQ$-Fano $3$-folds and it plays a key role in the classification of smooth Fano $3$-folds (see \cite{IP99}). 
On the other hand, the classification of terminal/canonical weak $\bQ$-Fano $3$-folds is a wildly open problem and very few results are known (for instance \cite{Takagi02}).
So it is quite crucial to understand the behavior of anti-canonical volumes of terminal/canonical weak $\bQ$-Fano $3$-folds.

It was proved in \cite[Theorem~1.1]{CC08} that for a canonical weak $\bQ$-Fano $3$-fold $X$, $(-K_{X})^3\geq \frac{1}{330}$, and this lower bound is optimal. On the other hand, for the upper bound, there are only a few partial results:
\begin{enumerate}
\item {$(-K_{X})^3\leq 6^3\cdot (24!)^2$} if $X$ is a terminal weak $\mathbb{Q}$-Fano $3$-fold whose anti-canonical map is small (\cite[Theorem~1.2]{KMMT00});
\item {$(-K_{X})^3\leq 64$} if $X$ is a Gorenstein terminal $\mathbb{Q}$-Fano $3$-fold (\cite{Nami});
\item {$(-K_{X})^3\leq 72$} if $X$ is a Gorenstein canonical $\bQ$-Fano $3$-fold (\cite[Theorem~1.5]{Prok05});
\item {$(-K_{X})^3\leq \frac{125}{2}$} if $X$ is a non-Gorenstein $\mathbb{Q}$-factorial terminal $\bQ$-Fano $3$-fold with $\rho(X)=1$ (\cite[Theorem~1.2]{Prok07});
\item {$(-K_{X})^3\leq 72$} if $X$ is a $\mathbb{Q}$-factorial terminal weak $\bQ$-Fano $3$-fold with $\rho(X)=2$ except in one case with {$(-K_{X})^3\leq 81$} (\cite{Lai21}).
\end{enumerate}
Here (2)--(4) are optimal, but the methods 
essentially rely on the Gorenstein condition or the Picard rank condition, so they
could not be easily applied to an arbitrary terminal/canonical $\bQ$-Fano $3$-fold. For a terminal/canonical weak $\bQ$-Fano $3$-fold, it was conjectured by Prokhorov that the upper bound for the anti-canonical volume should also be $72$, but even an explicit upper bound is not established yet in the literature. 

As the main result of this paper, we provide an effective upper bound for the anti-canonical volume of a canonical weak $\bQ$-Fano $3$-fold.
\begin{thm}\label{mainthm}
Let $X$ be a canonical weak $\bQ$-Fano $3$-fold. Then $$(-K_{X})^3\leq 324.$$
\end{thm}
The following is a direct corollary by Reid's Riemann--Roch formula. It can be viewed as a weak generalization of \cite[Corollary~1.8]{Prok05}.
\begin{cor}\label{maincor}
Let $X$ be a canonical weak $\bQ$-Fano $3$-fold. Then $$h^0(X, -K_{X})\leq 164.$$
\end{cor}

The ideas of proof originate from \cite{Jiang14, Jiang18}, where the first author gave a general strategy on bounding anti-canonical volumes of $\bQ$-Fano varieties with prescribed singularities.
 The strategy works for canonical weak $\bQ$-Fano $3$-folds as well, but the issue is that the method in \cite{Jiang14, Jiang18} is not effective, that is, it only gives the existence of an upper bound rather than an explicit formula.
 So in order to give the desired upper bound, we have to provide a much more effective version of the method in \cite{Jiang14, Jiang18} for canonical weak $\bQ$-Fano $3$-folds.
 
 We briefly explain the sketch of the proof of 
 Theorem~\ref{mainthm}. 
 As the first step, by an MMP argument, we can construct a birational model $Y$ of $X$ with a fibration structure $Y\to S$, and reduce the problem to the boundedness of the anti-canonical volume of $Y$ (Proposition~\ref{prop MFS}). 
 As $X$ is canonical, the birational model $Y$ and the fibration $Y\to S$ have many geometric restrictions which lead to an effective upper bound for the anti-canonical volume. Then the problem splits into two cases: $\dim S=1$ or $\dim S=2$. When $\dim S=1$, by ideas in \cite{Jiang14, Jiang18}, to give an upper bound for the 
 anti-canonical volume of $Y$, it suffices to give a lower bound for certain log canonical thresholds of $F$, where $F$ is a general fiber of $Y\to S$. The proof uses the connectedness lemma to construct non-klt centers on $F$ (Proposition~\ref{prop dp fibration}), and the effective lower bound of log canonical thresholds is treated in Section~\ref{sec 3} (Theorem~\ref{thm SLCT}).
 When $\dim S=2$, a similar argument can be used to construct non-klt centers on $F$, but as $S$ is a surface, we need to consider a well-chosen base point free linear system on $S$ (Proposition~\ref{prop conic bdl}).

This paper is organized as the following. In Section~\ref{sec 2}, we introduce definitions and basic knowledge. In Section~\ref{sec 3}, we study the lower bound of certain log canonical thresholds on smooth weak del Pezzo surfaces (Theorem~\ref{thm SLCT}). In Section~\ref{sec 4}, we prove Theorem~\ref{mainthm}.

\section{Preliminaries}\label{sec 2}
We adopt standard notation and definitions in \cite{KM} and will freely use them.
We use $\mathbb{F}_n$ to denote the $n$-th Hirzebruch surface
$\mathbb{P}_{\mathbb{P}^1}(\mathcal{O}_{\mathbb{P}^1}\oplus\mathcal{O}_{
\mathbb{P}^1}(n))$.
We use $\sim_\bQ$ and $\sim_\bR$ to denote the $\bQ$-linear equivalence and $\bR$-linear equivalence respectively.

\subsection{Singularities of pairs}
\begin{definition}
A {\it pair} $(X, B)$ consists of a normal variety $X$ and an effective
$\bR$-divisor $B$ on $X$ such that
$K_X+B$ is $\bR$-Cartier.
\end{definition}

\begin{definition}\label{def sing}
Let $(X, B)$ be a pair. Let $f: Y\to X$ be a log
resolution of $(X, B)$, write
$$
K_Y =f^*(K_X+B)+\sum a_iE_i,
$$
where $E_i$ are distinct prime divisors on $Y$ satisfying $f_*(\sum a_iE_i)=-B$. The number $a_i$ is called the {\it discrepancy} of $E_i$ with respect to $(X, B)$, and is denoted by $a(E_i, X, B)$. 
The pair $(X,B)$ is called
\begin{itemize}
\item[(a)] \emph{Kawamata log terminal} ({\it klt},
for short) if $a_i>-1$ for all $i$;

\item[(b)] \emph{log canonical} (\emph{lc}, for
short) if $a_i\geq -1$ for all $i$;

\item[(c)] \emph{terminal} if $a_i> 0$ for all $f$-exceptional divisors $E_i$ and for all $f$;

\item[(d)] \emph{canonical} if $a_i\geq 0$ for all $f$-exceptional divisors $E_i$ and for all $f$.
\end{itemize}
Usually, we write $X$ instead of $(X,0)$ in the case when $B=0$.

Given a closed point $P\in X$, we say that $(X, B)$ is \emph{lc near $P$} (resp. \emph{klt near $P$}) if there exists an open neighborhood $U$ of $P$ such that $(U, B|_U)$ is lc (resp. klt). 
\end{definition}

The following lemma is a numerical criterion for local singularities.
\begin{lem}[{cf. \cite[Proposition~9.5.13]{Positivity2}}]\label{ls}
Let $(X, B)$ be a pair. Suppose that $P\in X$ is a smooth closed point such that $\mult_P B\leq 1$ (resp. $<1$). 
Then $(X, B)$ is lc near $P$ (resp. klt near $P$).
\end{lem}

Here for an effective Cartier divisor $D$ on $X$, $\mult_P D$ denotes the multiplicity of the divisor $D$ at $P$, i.e. the vanishing order at $P$ of a local equation for $D$. It can be extended to effective $\bR$-Cartier $\bR$-divisors by homogeneous property.

\subsection{Non-klt centers}

\begin{definition}
In the setting of Definition~\ref{def sing}, $E_i$ is called a {\it non-klt place} of $(X, B)$ if $a_i\leq -1$.
A proper subvariety $V\subset X$ is called a {\it non-klt center} of $(X, B)$ if it is the image of a non-klt place. The {\it non-klt locus} $\text{Nklt}(X, B)$ of $(X, B)$ is the union of all non-klt centers of $(X, B)$.
\end{definition}

In practice, the following lemma is often used to construct non-klt centers.
\begin{lem}[{cf. \cite[Lemma~2.29]{KM}}]\label{lem mult>k nonklt}
Let $(X, B)$ be a pair and $W\subset X$ be a closed subvariety of codimension $k$ such that $W$ is not contained in the singular locus of $X$. If $\mult_W B\geq k$, then $W$ is a non-klt center of $(X, B)$.
\end{lem}

Here $\mult_W B$ is the multiplicity of $B$ at the generic point of $W$.

Recall the following special case of the Shokurov--Koll\'ar connectedness lemma.
\begin{lem}[{cf. \cite{Shokurov}, \cite[Theorem~17.4]{Kol92}}]\label{CL}Let $(X,B)$ be a pair with $X$ proper. If $-(K_X+B)$ is nef and big, then ${\rm Nklt} (X,B)$ is connected.
\end{lem}
\subsection{Log canonical thresholds}
\begin{definition}

Let $(X, B)$ be a pair which is lc near a closed point $P\in X$. Let $D\neq 0$ be an effective $\bR$-Cartier $\bR$-divisor on $X$ in a neighborhood of $P$. The {\it log canonical threshold} of $D$ with respect to $(X\ni P,B)$ is
$$\lct(X\ni P, B; D) = \sup\{t\in \bR \mid (X, B+ tD) \text{ is lc near } P\}.$$
If $B=0$, then we write $\lct(X\ni P; D)$ instead of $\lct(X\ni P, B; D)$. 
\end{definition}

It is well-known that log canonicity and log canonical thresholds satisfy convexity.
\begin{lem}[cf. {\cite[Lemma~3.8]{HLQ21}, \cite[Lemma~2.18]{HJL21}}]\label{Convexity for lct}
Let $P\in X$ be a closed point on a normal variety. Let $({X},{B_i})$ be a pair for $1\leq i\leq m$, $C\neq0$ an effective $\bR$-Cartier $\bR$-divisor on ${X}$ in a neighborhood of $P$, $\lambda_i$ a non-negative real number for $1\leq i\leq m$ such that $\sum_{i=1}^m\lambda_i= 1$. 
Suppose that $({X},{B_i})$ is lc near $P$ for all $1\leq i\leq m$. Then
\begin{enumerate}
 \item\label{lct convex 1} $({X},\sum_{i=1}^m \lambda_i {B_i})$ is lc near $P$;
 
 \item\label{lct convex ineq} the following inequality holds:
$$\lct\Big({X}\ni P,\sum_{i=1}^m \lambda_i {B_i};{C}\Big)\geq \sum_{i=1}^m \lambda_i \lct({X}\ni P,{B_i};{C});$$ 
\item\label{lct convex not klt} if $\lambda_{i_0}\neq 0$ and $({X},{B_{i_0}})$ is klt near $P$ for some $1\leq i_0\leq m$, then $({X},\sum_{i=1}^m \lambda_i {B_i})$ is klt near $P$.
\end{enumerate}

\end{lem}

\subsection{Volumes}
\begin{definition}
Let $X$ be an $n$-dimensional projective variety and $D$ be a Cartier divisor on $X$. The {\it volume} of $D$ is the real number
$$
\Vol(X, D)=\limsup_{m\to \infty}\frac{h^0(X, \OO_X(mD))}{m^n/n!}.
$$
For more details and properties of volumes, we refer to \cite[2.2.C]{Positivity1} and \cite[11.4.A]{Positivity2}. Moreover, by homogeneous property of volumes, the definition can be extended to $\bQ$-Cartier $\bQ$-divisors. Note that if $D$ is a nef $\bQ$-Cartier $\bQ$-divisor, then $\Vol(X, D)=D^n$. 
\end{definition}

\subsection{Intersection multiplicities on smooth surfaces}

Let $S$ be a smooth surface and $P\in S$ a closed point.
Assume that $C$ and $D$ are curves on $S$ having no common irreducible component, then we can define the {\it intersection multiplicity} $$(C\cdot D)_P=\length( \mathcal{O}_{P, S}/(f, g))$$
where $f, g$ are local equations of $C, D$ at $P$ (\cite[Exercise~I.5.4]{H}).
By homogeneous property, the definition of {\it intersection multiplicity} $(C\cdot D)_P$ can be extended for effective $\bR$-divisors
 $C$ and $D$ on $S$ having no common irreducible component.
 Under this setting,
it is well-known that 
 $(C\cdot D)_P\geq \mult_P C\cdot \mult_P D$ by \cite[Exercise~I.5.4]{H}.
Moreover, if $S$ is projective, then $(C\cdot D)_P\leq (C\cdot D)$ by \cite[Proposition~V.1.4]{H}.

\section{Log canonical thresholds on weak del Pezzo surfaces}\label{sec 3}
In this section, we investigate lower bounds of log canonical thresholds on weak del Pezzo surfaces. First, we prove 2 useful lemmas on (local) log canonical thresholds, of which we got the ideas during the preparation of \cite{HJL21} (cf. \cite[Theorem~1.11]{HJL21}).

\begin{lem}\label{lem lct m=I=1}
Let $S$ be a smooth surface and $P\in S$ a closed point. Let $B, C$ be effective $\mathbb{R}$-divisors on $S$ having no common irreducible component. Suppose that $\mult_P B= (B\cdot C)_P=1$. Then $(S, B+C)$ is lc near $P$.
\end{lem}

Here in Lemma~\ref{lem lct m=I=1} we do not assume that $\Supp C$ is smooth as in \cite[Theorem~1.11]{HJL21}, so these 2 results do not contain each other.
Also, it appears to us that the assumptions in Lemma~\ref{lem lct m=I=1} are quite simple, so the statement might be known to experts, 
but we could not find any similar statement in the literature (cf. \cite[Corollary~6.46]{KSC04}).

\begin{proof}
As being lc is a closed condition on coefficients, after slightly perturbing the coefficients of $B$ and $C$, we may assume that $B$ and $C$ are $\mathbb{Q}$-divisors.
Take a positive integer $k$ such that $kB$ and $kC$ are Cartier divisors.
Fix a local coordinate system $(x, y)$ for $S$ at $P$, and suppose that $(f_B= 0)$ and $(f_C=0)$ are local equations of $kB$ and $kC$ in the coordinates $(x, y)$ respectively for some $f_B, f_C\in \mathbb{C}[x, y]$.
Fix positive integers $w(x)$ and $w(y)$ as weights of $x$ and $y$ respectively, recall that for $f\in \mathbb{C}[x, y]$, $\text{mult}_w(f)$ is the weight of the lowest weight term of $f$.

By \cite[Theorem~6.40]{KSC04}, to show that $(S, B+C)$ is lc near $P$, it suffices to show that
\begin{align}
\frac{w(x)+w(y)}{\text{mult}_w(f_B)+\text{mult}_w(f_C)}\geq \frac{1}{k}.\label{w+w/m+m}
\end{align}
Without loss of generality, we may assume that $w(x)\leq w(y)$.
Note that
$\mult_P B=1$ implies that there is a monomial $x^iy^j$ with non-zero coefficient in $f_B$ such that $i+j=k$.
So $$\mult_w (f_B)\leq i\cdot w(x)+j\cdot w(y)\leq k\cdot w(y).$$
On the other hand, by the definition of $\mult_w (f_B)$, there is a monomial $x^{i'}y^{j'}$ with non-zero coefficient in $f_B$ such that $$\mult_w (f_B)= i'\cdot w(x)+j'\cdot w(y)\geq k\cdot w(x),$$ where we used the fact that $i'+j'\geq k$ as
$\mult_P B=1$.
So in summary, 
\begin{align}
k\cdot w(x)\leq \mult_w (f_B)\leq k\cdot w(y).\label{wx<wf<wy}
\end{align}
Moreover, by \cite[Lemma~6.47]{KSC04}, 
\begin{align}
\frac{\text{mult}_w(f_B)\cdot\text{mult}_w(f_C)}{w(x)\cdot w(y)}\leq (kB\cdot kC)_P=k^2.\label{mm/ww}
\end{align}
By \eqref{wx<wf<wy} and \eqref{mm/ww},
\begin{align*}
{}& k\cdot w(x)+k\cdot w(y)-\text{mult}_w(f_B)-\text{mult}_w(f_C)\\
\geq {}& k\cdot w(x)+k\cdot w(y)-\text{mult}_w(f_B)- \frac{k^2\cdot w(x)\cdot w(y)}{\text{mult}_w(f_B)}\\
= {}& \frac{(\text{mult}_w(f_B)-k\cdot w(x))(k\cdot w(y)-\text{mult}_w(f_B))}{\text{mult}_w(f_B)}\geq 0.
\end{align*}
Hence \eqref{w+w/m+m} is proved.
\end{proof}
The following is a direct consequence of Lemma~\ref{lem lct m=I=1}.
\begin{lem}\label{lct lower bound}
Let $S$ be a smooth surface and $P\in S$ a closed point. Let $B, C$ be effective $\mathbb{R}$-divisors on $S$ having no common irreducible component. Denote $m_B=\mult_P B$, $I=(B\cdot C)_P$. Suppose that $m_B>0$ and $I>0$. Then $$\Big(S, \frac{1}{m_B}B+\frac{m_B}{I} C\Big)$$ is lc near $P$. 
\end{lem}

The following is the main theorem of this section.
\begin{thm}\label{thm SLCT}
Let $S$ be a smooth weak del Pezzo surface.
Suppose that there exists 
a real number $0<t<1$ and effective $\bR$-divisors $B, D$ on $S$ such that
\begin{enumerate}
\item $(S, B)$ is canonical;
 \item $(S, (1-t)B+tD)$ is not klt;
 \item $B\sim_{\bR} D\sim_{\bR}-K_S$;
 \item no irreducible component of $\Supp B$ has negative self-intersection;
 \item $D$ and $B$ have no common irreducible component.
\end{enumerate}
Then $t\geq \frac{1}{6}$.
Furthermore, if there is a birational morphism $S\to \mathbb{F}_0$, then $t\geq \frac{1}{5}$.
\end{thm}

\begin{proof}
If $\sigma: S\to T$ is a birational morphism to another smooth weak del Pezzo surface $T$, then by the negativity lemma \cite[Lemma~3.39]{KM},
\begin{align*}
K_S+B={}&\sigma^*(K_T+\sigma_*B), \\
K_S+(1-t)B+tD={}&\sigma^*(K_T+(1-t)\sigma_*B+t\sigma_*D).
\end{align*}
The assumption on $\Supp B$ implies that no irreducible component of $\Supp B$ is contracted by $\sigma$ and no irreducible component of $\Supp (\sigma_*B)$ has negative self-intersection. So $(T, \sigma_*B)$ is canonical and $(T, (1-t)\sigma_*B+t\sigma_*D)$ is not klt.
Hence $T, \sigma_*B, \sigma_*D, t$ satisfy all assumptions of the theorem. 
So in order to give a lower bound of $t$, we may replace $S, B, D$ by $T, \sigma_*B, \sigma_*D$. 
By the classification of smooth weak del Pezzo surfaces (or standard MMP argument, see \cite[Lemma~3.1]{Jiang13} or \cite[Theorem~8.1.15]{Dol12}), every smooth weak del Pezzo surface has a birational morphism to $\bP^2$ or $\mathbb{F}_0$ or $\mathbb{F}_2$.
So without loss of generality, we may assume that $S=\bP^2$ or $\mathbb{F}_0$ or $\mathbb{F}_2$.

By assumption, $\Nklt(S,(1-t)B+tD)\neq \emptyset$.
Fix a closed point $P\in \Nklt(S,(1-t)B+tD)$. 
Then $t\geq \lct(S\ni P, (1-t)B; D)$. 
Set $I=(B\cdot D)_P$ and $m_B=\mult_P B$. Since $(S, B)$ is canonical, $m_B\leq 1$ by \cite[Theorem~4.5]{KM}. Thus $P\in \Supp D$ by Lemma~\ref{ls}.

If $m_B=0$, then $$t\geq \lct(S\ni P, (1-t)B; D)=\lct(S\ni P; D)\geq \frac{1}{4}$$ by \cite[Theorem~1.7]{CH08} and \cite[Corollary~5.2]{CS08}. In fact, if  $S=\bP^2$ or $\mathbb{F}_0$, then by \cite[Theorem~1.7]{CH08}, $\lct(S\ni P; D)\geq  \lct(S)\geq \frac{1}{3}$; if  $S=\mathbb{F}_2=\mathbb{P}_{\mathbb{P}^1}(\mathcal{O}\oplus \mathcal{O}(-2))$, then by \cite[Corollary~5.2]{CS08}, $\lct(S\ni P; D)\geq  \lct(S)\geq \frac{1}{4}$. Here $\lct(S)$ is defined to be $$\lct(S)=\sup\{t\in \bR \mid (S, tD') \text{ is lc for all effective } \mathbb{Q}\text{-divisor } D'\sim_\mathbb{Q}-K_S \}.$$

From now on, we suppose that $m_B>0$ and $I>0$, then
by 
Lemma~\ref{lct lower bound} and Lemma~\ref{Convexity for lct}\eqref{lct convex ineq},
\begin{align}
t\geq {}&\lct(S\ni P, (1-t)B; D)\notag\\
\geq {}& (1-t)m_B\lct\Big(S\ni P, \frac{1}{m_B} B; D\Big)+(1-(1-t)m_B)\lct(S\ni P, 0; D)\notag\\
\geq {}& (1-t)\frac{m_B^2}{I}+(1-(1-t)m_B)\lct(S\ni P; D).\label{ineq of t}
\end{align}

We will use the following 2 direct consequences of \eqref{ineq of t}.
\begin{claim}\label{claim I lct}
\begin{enumerate}
\item \label{claim I lct part 1} If $I\leq 9$ and $\lct(S\ni P; D)\geq \frac{2}{3}$, then $t\geq \frac{1}{4}.$
\item \label{claim I lct part 2} If $I\leq 8$ and $\lct(S\ni P; D)\geq \frac{1}{2}$, then $t\geq \frac{1}{5}.$
\end{enumerate}
\end{claim}
\begin{proof}
For \eqref{claim I lct part 1}, \eqref{ineq of t} implies that
\begin{align*}
t\geq \frac{(1-t)(m_B^2-6m_B)}{9}+\frac{2}{3}
 \geq \frac{-5(1-t)}{9}+\frac{2}{3},
\end{align*} where the second inequality follows from $m_B\leq 1$. Hence $t\geq \frac{1}{4}.$

Similarly, for \eqref{claim I lct part 2}, \eqref{ineq of t} implies that
\begin{align*}
 t\geq \frac{(1-t)(m_B^2-4m_B)}{8}+\frac{1}{2}
 \geq \frac{-3(1-t)}{8}+\frac{1}{2},
\end{align*} and hence $t\geq \frac{1}{5}.$
\end{proof}

Now we proceed to the proof of the theorem for $S=\bP^2$ or $\mathbb{F}_0$ or $\mathbb{F}_2$ case by case.

\medskip

{\bf Case 1}: $S=\mathbb{F}_{0}$.

\medskip

In this case, $I\leq (-K_S)^2=8$ and $\lct(S\ni P; D)\geq \frac{1}{2}$ by \cite[Theorem~1.7]{CH08}. Then by Claim~\ref{claim I lct}\eqref{claim I lct part 2}, $t\geq \frac{1}{5}.$

\medskip

{\bf Case 2}: $S=\bP^2$. 

\medskip

Write the prime decomposition $D=\sum_i d_iD_i$. As $\rho(S)=1$, for any $i$, $D_i\sim_{\bQ}c_i(-K_S)$ for some positive rational number $c_i$. Since $D\sim_{\bR}-K_S$, $\sum_i d_ic_i=1$. Then by Lemma~\ref{Convexity for lct}\eqref{lct convex not klt}, there exists an index $i$ such that $(S, (1-t)B+t\frac{D_i}{c_i})$ is not klt near $P$. So after replacing $D$ by $\frac{D_i}{c_i}$, we may assume that $\Supp D$ is irreducible.

If $\Supp D$ is a line $L$ on $S$, then $D=3L$. Then by \cite[Theorem~{1.11}]{HJL21},
\begin{align*}
3t\geq {}&\lct(S\ni P, (1-t)B; L)\\
\geq {}&\min\Big\{1, 1+\frac{m_B}{(B\cdot L)_P}-(1-t)m_B\Big\}\\
\geq {}&\min\Big\{1, 1-\frac{2m_B}{3}+tm_B\Big\}.
\end{align*}
Here the last inequality follows from $(B\cdot L)_P\leq (-K_S\cdot L)= 3$. Recall that $m_B\leq 1$, then by easy calculation, $t\geq \frac{1}{6}$. In fact, suppose that $t< \frac{1}{3}$, then $3t\geq 1-\frac{2m_B}{3}+tm_B$, which implies that $t\geq \frac{3-2m_B}{9-3m_B}\geq \frac{1}{6}$.

Now suppose that $\Supp D$ is not a line on $S$. Take $c=\lct(S\ni P; D)$ and denote $m_D=\mult_{P}D$. Then by Lemma~\ref{ls}, $cm_D\geq 1$.
Let $\pi : S'\to S$ be the blow-up of $S$ at $P$. 
Then
$$
K_{S'}+cD'+(cm_D-1)E=\pi^*(K_{S}+cD),
$$
where $D'$ is the strict transform of $D$ on $S'$ and $E$ is the exceptional divisor.
For any closed point $Q\in E$, take $F_Q$ to be the natural ruling of $S'\simeq \mathbb{F}_1$ passing $Q$, then by the assumption that $\Supp D$ is not a line, $F_Q\not \subset \Supp D'\cup E$. Hence $$((cD'+(cm_D-1)E)\cdot F_Q)_Q\leq ((cD'+(cm_D-1)E)\cdot F_Q)=3c-1.$$
If $c< \frac{2}{3}$, then by \cite[Corollary~5.57]{KM} or Lemma~\ref{ls}, $({S'}, cD'+(cm_D-1)E)$ is klt near any closed point $Q\in E$, which implies that $(S, cD)$ is klt near $P$, but it contradicts the definition of $c$.
So $\lct(S\ni P; D)=c\geq \frac{2}{3}$. As $I\leq (-K_S)^2=9$, by Claim~\ref{claim I lct}\eqref{claim I lct part 1}, $t\geq \frac{1}{4}.$

\medskip

{\bf Case 3}: $S=\mathbb{F}_{2}$.

\medskip

Denote by $F$ the natural ruling passing $P$ and $E_0$ the negative section of $\mathbb{F}_{2}$.
Recall that $(K_S\cdot F)=(E_0^2)=-2$, $(F^2)=(K_S\cdot E_0)=0$, $(F\cdot E_0)=1$.

First we reduce to the following 2 cases: 
\begin{enumerate}[label=(\roman*)]
 \item $D=4F+2E_0$, or
 \item $F\not \subset \Supp D$.
\end{enumerate}
Write $D=G+bF+eE_0$ where $F, E_0$ are not contained in $\Supp G$.
Note that $D\sim_\bR-K_S\sim_\bQ 4F+2E_0$. 
Therefore, 
$$
0\leq (G\cdot F)=((D-bF-eE_0)\cdot F)=2-e
$$
and
$$
0\leq (G\cdot E_0)=((D-bF-eE_0)\cdot E_0)=-b+2e,
$$
which imply that $2\geq e\geq \frac{b}{2}\geq 0$.
Denote $D_1=4F+2E_0$.
If $b=4$, then $D=D_1$.
If $b=0$, then $F\not \subset \Supp D$.
If $0<b<4$, denote $D_2=\frac{4}{4-b}(D-\frac{b}{4}D_1)$,
then $D_1,D_2$ are effective $\bR$-divisors with $D=\frac{b}{4}D_1+(1-\frac{b}{4})D_2$ and $D_1\sim_\bR D_2\sim_\bR -K_S$. 
Hence by Lemma~\ref{Convexity for lct}\eqref{lct convex not klt}, there exists an index $i=1$ or $2$ such that $(S, (1-t)B+tD_i)$ is not klt near $P$. So after replacing $D$ by $D_1$ or $D_2$, we may assume that either $D=D_1$, or $F\not \subset \Supp D$.

If $F\not \subset \Supp D$, then $(D\cdot F)_P\leq (D\cdot F)=2$. So \cite[Corollary~5.57]{KM} or Lemma~\ref{ls} implies that 
$\lct(S\ni P; D)\geq \frac{1}{2}$. As $I\leq (-K_S)^2=8$, by Claim~\ref{claim I lct}\eqref{claim I lct part 2}, $t\geq \frac{1}{5}.$

If $D=4F+2E_0$, then $E_0\not \subset \Supp B$ as $B$ and $D$ have no common irreducible component. On the other hand, $(B\cdot E_0)=(-K_S\cdot E_0)=0$. So $E_0\cap \Supp B=\emptyset.$ In particular, $P\not\in E_0$ and $P\in \Nklt(S,(1-t)B+4tF)$. 
 Then by \cite[Theorem~{1.11}]{HJL21},
\begin{align*}
4t\geq {}&\lct(S\ni P, (1-t)B; F)\\
\geq {}&\min\Big\{1, 1+\frac{m_B}{(B\cdot F)_P}-(1-t)m_B\Big\}\\
\geq {}&\min\Big\{1, 1-\frac{m_B}{2}+tm_B\Big\}.
\end{align*}
Here the last inequality follows from $(B\cdot F)_P\leq (-K_S\cdot F)= 2$. Recall that $m_B\leq 1$, then by easy calculation similar to Case 2, $t\geq \frac{1}{6}$.

Combining all the above cases, the proof is concluded.
\end{proof}

\section{Proof of the main theorem}\label{sec 4}
In this section, we prove Theorem~\ref{mainthm}.

\subsection{A reduction step}
First we reduce the problem to another birational model with more geometric structures.
The following proposition is an improvement of 
\cite[Theorem~4.1]{Jiang14} and \cite[Proposition~3.9]{CJ20} for canonical weak $\bQ$-Fano $3$-folds.

\begin{prop}\label{prop MFS}
Let $X$ be a canonical weak $\bQ$-Fano $3$-fold. 
Then $X$ is birational to a normal projective $3$-fold $Y$ satisfying the following:
\begin{enumerate}
 \item\label{prop MFS part1} $Y$ is $\mathbb{Q}$-factorial terminal; 
 \item\label{prop MFS part2} $(-K_X)^{3}\leq \Vol(Y, -K_{Y})$;
 \item\label{prop MFS part3} for any sufficiently large and divisible positive integer $n$, $|-nK_Y|$ is movable;
 \item\label{prop MFS part4} for a general member $M\in |-nK_Y|$, $M$ is irreducible and $(Y, \frac{1}{n}M)$ is canonical;
 \item\label{prop MFS part5} there exists a projective morphism $\pi: Y\to S$ with connected fibers where $F$ is a general fiber of $\pi$, such that one of the following conditions holds:
 \begin{enumerate}
 \item $S$ is a point and $Y$ is a $\bQ$-Fano $3$-fold with $\rho(Y)=1$;
 \item $S= \mathbb{P}^1$ and $F$ is a smooth weak del Pezzo surface;
 \item $S$ is a del Pezzo surface with at worst Du Val singularities and $\rho(S)=1$, and $F\simeq \mathbb{P}^1$.
 \end{enumerate}

\end{enumerate}
\end{prop}
\begin{proof}
Let $X$ be a canonical weak $\bQ$-Fano $3$-fold. By \cite[Theorem~6.23, Theorem~6.25]{KM}, we can take a $\bQ$-factorial terminalization of $X$, that is, there is a projective birational morphism $\phi: W \to X$ such that $K_{W}=\phi^*(K_X)$ and $W$ is $\mathbb{Q}$-factorial terminal.
In particular, $W$ is a $\mathbb{Q}$-factorial terminal weak $\bQ$-Fano $3$-fold.
As $-K_W$ is big, we can run a $K$-MMP on $W$ which ends up with a Mori fiber space $f: Y\to T$ (\cite[Corollary~1.3.2]{BCHM}), where 
$Y$ is $\bQ$-factorial terminal, $-K_Y$ is ample over $T$, $\dim Y>\dim T$, $f_*\OO_Y=\OO_T$, and $\rho(Y/T)=1$.

Take a common resolution $p:V\to W$, $q:V\to Y$. Then 
$$
p^*K_W=q^*K_Y+E_0
$$
where $E_0$ is an effective $\mathbb{Q}$-divisor on $V$ exceptional over $Y$.
Hence 
\begin{align*}
 \Vol(Y, -K_Y)={}&\Vol(V, -q^*K_Y)\\={}&\Vol(V, -p^*K_W+E_0)\\\geq {}& \Vol(V, -p^*K_W)=(-K_X)^3. 
\end{align*}
This concludes \eqref{prop MFS part1}\eqref{prop MFS part2}.

As $W$ is terminal and $-K_W$ is nef and big, $-K_W$ is semi-ample by the base point free theorem (\cite[Theorem~3.3]{KM}). 
So for any sufficiently large and divisible positive integer $n$, $|-nK_W|$ is base point free. Hence for a general member $M_W\in |-nK_W|$, $M_W$ is irreducible and $(W, \frac{1}{n}M_W)$ is canonical by the Bertini theorem (cf. \cite[Lemma~5.17]{KM}).
Denote $M$ to be the strict transform of $M_W$ on $Y$, which is a general member in $|-nK_Y|$, then 
$$
p^*\Big(K_W+\frac{1}{n}M_W\Big)=q^*\Big(K_Y+\frac{1}{n}M\Big)\sim_\bQ 0
$$
by the negativity lemma \cite[Lemma~3.39]{KM}. This implies that for any prime divisor $E$ over $Y$, $a(E, W, \frac{1}{n}M_W)=a(E, Y, \frac{1}{n}M)$. As $M_W$ is irreducible, it is not exceptional over $Y$. So for any prime divisor $E$ which is exceptional over $Y$, either $E$ is exceptional over $W$, or $E$ is a prime divisor on $W$ such that $E\neq M_W$. In either case, $a(E, Y, \frac{1}{n}M)=a(E, W, \frac{1}{n}M_W)\geq 0$. Hence $(Y, \frac{1}{n}M)$ is canonical. This concludes \eqref{prop MFS part3}\eqref{prop MFS part4}.

Finally, we show \eqref{prop MFS part5} by discussing the dimension of $T$. Note that $\dim T\in\{0,1,2\}$. Recall that by \cite{KMMT00} or \cite[Theorem~1]{ZQ06}, $W$ is rationally connected, which implies that $Y$ and $T$ are also rationally connected.

If $\dim T=0$, then take $S=T$ and $Y$ is a $\bQ$-Fano $3$-fold with $\rho(Y)=1$. In this case, we get (a).

If $\dim T=1$, then $T\simeq \mathbb{P}^1$. Take $S=\mathbb{P}^1$ and $\pi$ the induced morphism. Then the general fiber $F$ of $\pi: Y\to S$ is a smooth del Pezzo surface as $-K_Y$ is ample over $S$ and $Y$ is terminal (\cite[Corollary~5.18]{KM}). In this case, we get (b).

If $\dim T=2$, then $T$ is a rational surface as it is rationally connected, and $T$ has at worst Du Val singularities by \cite[Theorem~1.2.7]{MP}. 
We can run a $K$-MMP on $T$ which ends up with a surface $T'$, which is either a del Pezzo surface with at worst Du Val singularities and $\rho(T')=1$, or there is a morphism $T'\to \mathbb{P}^1$. 
In the former case, take $S=T'$ and $\pi: Y\to S$ to be the induced morphism $Y\to T\to T'$, then the general fiber of $\pi$ is a smooth rational curve as $-K_Y$ is ample over $T$. In this case, we get (c).
In the latter case, take $S=\mathbb{P}^1$ and $\pi: Y\to S$ to be the induced morphism $Y\to T\to T'\to \mathbb{P}^1$.
Recall that $-K_Y$ is big and $|-nK_Y|$ is movable, hence for a general fiber $F$ of $\pi$, $-K_F$ is big and $|-nK_F|$ is movable. This implies that $-K_F$ is nef and big as $\dim F=2$. Also $F$ is smooth as $Y$ is terminal (\cite[Corollary~5.18]{KM}). So $F$ is a smooth weak del Pezzo surface. In this case, we get (b).
\end{proof}
\subsection{Weak del Pezzo fibration case}

In this subsection, we treat the case in Proposition~\ref{prop MFS}(b).
\begin{prop}\label{prop dp fibration}
Keep the same notation as in Proposition~\ref{prop MFS}.
Suppose that $Y$ satisfies (b). Then $ \Vol(Y, -K_{Y})\leq 324.$
\end{prop}

\begin{proof}
 By assumption, a general fiber $F$ of $\pi: Y\to \mathbb{P}^1$ is a smooth weak del Pezzo surface. It is well-known that $K^2_F\leq 9$ (for example, it follows from the fact that $F$ has a birational morphism to one of $\mathbb{F}_0, \mathbb{F}_2, \mathbb{P}^2$ by \cite[Theorem~8.1.15]{Dol12}). 

To conclude the proof, we claim that $$\Vol(Y, -K_Y)\leq 36K_F^2.$$
Assume to the contrary that $\Vol(Y, -K_Y)> 36K_F^2,$
then we can find a rational number $s$ such that
$$\Vol(Y, -K_Y)>3sK_F^2>36K_F^2.$$
Here note that $s>12$.
Then by \cite[Lemma~2.5]{Jiang18},
\begin{align*}
 \Vol(Y, -K_Y-sF)\geq \Vol(Y, -K_Y)-3s\Vol(F, -K_F)>0.
\end{align*}
Hence there exists an effective $\bQ$-divisor $D\sim_\bQ-K_Y-sF$ on $Y$.

Fix a sufficiently large and divisible $n$ and take $M\in|-nK_Y|$ as in Proposition~\ref{prop MFS}\eqref{prop MFS part4}. Then $(Y, \frac{1}{n}M)$ is canonical and $K_Y+\frac{1}{n}M\sim_\bQ 0$.
Note that $-K_Y$ is big, so we may write $-K_Y=A+E$ where $A$ is an ample $\bQ$-divisor and $E$ is an effective $\bQ$-divisor on $Y$. Fix a sufficiently small rational number $\epsilon>0$, denote ${B_{\epsilon}}=\frac{1-\epsilon}{n}M+\epsilon E$. 
For two general fibers $F_1$ and $F_2$ of $\pi$, consider the pair
$$\Big(Y, \Big(1-\frac{2}{s}\Big){B_{\epsilon}} +\frac{2}{s}D+F_1+F_2\Big).$$
Note that
\begin{align*}
 {}&-\Big(K_Y+\Big(1-\frac{2}{s}\Big){B_{\epsilon}} +\frac{2}{s}D+F_1+F_2\Big)\\
 \sim_\bQ{}& -\Big(1-\frac{2}{s}\Big)(K_Y+{B_{\epsilon}})\sim_\bQ \Big(1-\frac{2}{s}\Big) \epsilon A
\end{align*}
is ample as $s>12$. Then by Lemma~\ref{CL}, $\Nklt(Y, (1-\frac{2}{s}){B_{\epsilon}} +\frac{2}{s}D+F_1+F_2)$ is connected. By construction, $F_1\cup F_2\subset \Nklt(Y, (1-\frac{2}{s}){B_{\epsilon}} +\frac{2}{s}D+F_1+F_2)$. 
Hence there is a non-klt center connecting $F_1$ and $F_2$. In particular, the above non-klt locus dominates $\bP^1$. Restricting on a general fiber $F$ of $\pi$, by the inversion of adjunction (see \cite[Lemma~5.50]{KM} or \cite[Theorem~2.11]{Jiang18}), 
$(F, (1-\frac{2}{s}){B_{\epsilon}}|_F+\frac{2}{s}D|_F)$ is not klt. As being klt is an open condition on coefficients, by the arbitrariness of $\epsilon$,
 it follows that $(F, (1-\frac{2}{s})\frac{1}{n}M|_F+\frac{2}{s}D|_F)$ is not klt for a very general fiber $F$ of $\pi$.

On the other hand, as $(Y, \frac{1}{n}M)$ is canonical, $(F, \frac{1}{n}M|_F)$ is canonical by the Bertini theorem (cf. \cite[Lemma~5.17]{KM}).
By the assumption on $M$, $M|_F$ is a general member of a movable linear system on $F$. So each irreducible component of $M|_F$ is nef, and $M|_F$ and $D|_F$ have no common irreducible component (recall that the construction of $D$ is independent of $M$). Also by construction, $\frac{1}{n}M|_F\sim_\bQ D|_F\sim_\bQ -K_F$. 
So we can apply Theorem~\ref{thm SLCT} to $F, \frac{1}{n}M|_F, D|_F$, which implies that $
\frac{2}{s}\geq \frac{1}{6}.$
But this contradicts the definition of $s$. Thus we conclude that
$\Vol(Y, -K_Y)\leq 36K_F^2\leq 324.$
\end{proof}

\subsection{Conic bundle case}

In this subsection, we treat the case in Proposition~\ref{prop MFS}(c).
\begin{prop}\label{prop conic bdl}
Keep the same notation as in Proposition~\ref{prop MFS}.
Suppose that $Y$ satisfies (c). Then $ \Vol(Y, -K_{Y})\leq 312.$
\end{prop}

\begin{proof}
Recall that by the classification of del Pezzo surfaces with Du Val singularities and Picard rank $1$ (see \cite{MZ88}, \cite[Remark~3.4(ii)]{Prok07}), there are 3 cases:
\begin{enumerate}
 \item $K_{S}^2=9$ and $S\simeq \mathbb{P}^2$;
 \item $K_{S}^2=8$ and $S\simeq \mathbb{P}(1, 1, 2)$;
 \item $1\leq K_{S}^2\leq 6$ and there exists a rational curve $C$ on $S$ such that $(K_S\cdot C)=-1$.
\end{enumerate}
Here as $\rho(S)=1$, $(K_S^2)\cdot C\sim_\bQ -K_S$ in Case (3).
Consider the linear system $\mathcal{H}$ on $S$ defined by
$$\mathcal{H}=
\begin{cases}
|\mathcal{O}_{\mathbb{P}^2}(1)| & \text{if } S\simeq \mathbb{P}^2;\\
|\mathcal{O}_{\mathbb{P}(1, 1, 2)}(2)| & \text{if } S\simeq \mathbb{P}(1, 1, 2);\\
|-K_{S}|& \text{if } 2\leq K_{S}^2\leq 6;\\
|-2K_{S}|& \text{if } K_{S}^2=1.
\end{cases}
$$
and fix an effective Weil divisor $C_0$ on $S$ defined by
$$C_0=
\begin{cases}
\text{a general element of } \mathcal{H} & \text{if } S\simeq \mathbb{P}^2 \text{ or }\mathbb{P}(1, 1, 2);\\
C& \text{if } 1\leq K_{S}^2\leq 6.
\end{cases}
$$

We claim that $\mathcal{H}$ is base point free and defines a generically finite map. The first two cases are easy facts. For the rest cases, consider $f: \tilde{S}\to S$ to be the minimal resolution of $S$, then $K_{\tilde{S}}=f^*K_{S}$ and $\tilde{S}$ is a smooth weak del Pezzo surface. Then the claim is equivalent to $f^*\mathcal{H}$ being base point free, which follows from the classification of smooth weak del Pezzo surfaces (see, for example, \cite[Theorem~8.3.2]{Dol12}).

Take a general element $H\in \mathcal{H}$ and denote $G=\pi^{-1}(H)=\pi^*H$. By the Bertini theorem, $G$ and $H$ are smooth.
Note that for a general fiber $F$ of $\pi|_G$, $F\simeq \mathbb{P}^1$, $(-K_G\cdot F)=2$, and $G|_G\sim (H^2)\cdot F$.

Define $d=(-K_S\cdot H)$ and define the number $c$ by the relation $C_0\sim_\bQ cH$. Then it is easy to compute that
$$d=
\begin{cases}
3 & \text{if } S\simeq \mathbb{P}^2;\\
4 & \text{if } S\simeq \mathbb{P}(1, 1, 2);\\
K_S^2 & \text{if } 2\leq K_{S}^2\leq 6;\\
2 & \text{if } K_{S}^2=1;
\end{cases}
$$
and
$$c=
\begin{cases}
1 & \text{if } S\simeq \mathbb{P}^2 \text{ or } \mathbb{P}(1, 1, 2);\\
\frac{1}{d} & \text{if } 1\leq K_{S}^2\leq 6.
\end{cases}
$$

First, we show that $\Vol(G, -K_{Y}|_G)\leq 4d$.
Note that $\pi|_G$ is factored through by a ruled surface over $H$, so $K_G^2\leq 8-8g(H)$.
Also note that by Proposition~\ref{prop MFS}\eqref{prop MFS part3}, $-K_Y|_G$ is nef. So
\begin{align*}
{}&\Vol(G, -K_{Y}|_G)=(-K_Y|_G)^2=(-K_{G}+G|_G)^2\\={}&K_{G}^2+4H^2\leq 8-8g(H)+4H^2=-4(K_S\cdot H)=4d.
\end{align*}
Here we used the adjunction formula $(K_S\cdot H)+H^2=2g(H)-2$. 

Next, we claim that $\Vol(Y, -K_{Y})\leq 24(2+c)d$.
Assume to the contrary that $\Vol(Y, -K_{Y})>24(2+c)d$.
Take a rational number $t$ such that
$$\Vol(Y, -K_Y)>12td>24(2+c)d.$$
Here note that $t>2(2+c)$. 
Then by \cite[Lemma~2.5]{Jiang18},
\begin{align*}
 \Vol(Y, -K_Y-tG)\geq \Vol(Y, -K_Y)-3t\Vol(G, -K_Y|_G)>0.
\end{align*}
Hence there exists an effective $\bQ$-divisor $D\sim_\bQ-K_Y-tG$ on $Y$.

Fix a sufficiently large and divisible $n$ and take $M\in |-nK_Y|$ as in Proposition~\ref{prop MFS}\eqref{prop MFS part4}, then $(Y, \frac{1}{n}M)$ is canonical and $K_Y+\frac{1}{n}M\sim_\bQ 0$.
Note that $-K_Y$ is big, so we may write $-K_Y=A+E$ where $A$ is an ample $\bQ$-divisor and $E$ is an effective $\bQ$-divisor on $Y$. Fix a sufficiently small rational number $\epsilon>0$, denote ${B_{\epsilon}}=\frac{1-\epsilon}{n}M+\epsilon E$. 
Since $\mathcal{H}$ is base point free, for a general point $z\in S\setminus C_0$, there exists a rational number $\eta>0$ (cf. \cite[4.8]{SOP}) such that for any general $H\in \mathcal{H}$ containing $z$,
\begin{align}
 {}& \Nklt\Big(Y, \Big(1-\frac{2+c}{t}\Big){B_{\epsilon}}+\frac{2+c}{t}D+\eta \pi^*(H)\Big)\notag\\
 ={}& \Nklt\Big(Y, \Big(1-\frac{2+c}{t}\Big){B_{\epsilon}}+\frac{2+c}{t}D\Big)\label{nklt eta H}.
\end{align}
Fix an integer $J>\frac{2}{\eta}$. As $\mathcal{H}$ defines a generically finite map, we may take distinct general elements $H_j\in \mathcal{H}$ containing $z$ for $1\leq j\leq J$ and take $G_1=\sum_{j=1}^J\frac{2}{J}\pi^*(H_j)$.
Denote  $F_z$ to be the fiber of $Y$ over $z$. Then $\mult_{F_z}G_1\geq 2$ and $G_1\sim_\bQ 2\pi^*(H)= 2G$. In particular, $F_z$ is a non-klt center of $(Y, G_1)$ by Lemma~\ref{lem mult>k nonklt}. So in a neighborhood of $F_z$,
\begin{align}
{}&\Nklt\Big(Y, \Big(1-\frac{2+c}{t}\Big){B_{\epsilon}}+\frac{2+c}{t}D+G_1+\pi^*C_0\Big)\notag\\
={}&\Nklt\Big(Y, \Big(1-\frac{2+c}{t}\Big){B_{\epsilon}}+\frac{2+c}{t}D+G_1)\notag\\
={}&\Nklt\Big(Y, \Big(1-\frac{2+c}{t}\Big){B_{\epsilon}}+\frac{2+c}{t}D\Big)\cup F_z,\label{nklt Fz}
\end{align}
here the first equality is by $z\not\in C_0$, and the second equality is by \eqref{nklt eta H} as they are the same away from $F_z$.
Now consider the pair $(Y, (1-\frac{2+c}{t}){B_{\epsilon}}+\frac{2+c}{t}D+G_1+\pi^*C_0)$. Recall that by definition, $\pi^*C_0\sim_\bQ c\pi^*(H)= cG$.
Then
\begin{align*}
{}&-\Big(K_{Y}+\Big(1-\frac{2+c}{t}\Big){B_{\epsilon}}+\frac{2+c}{t}D+G_1+\pi^*C_0\Big)\\\sim_\bQ{}&
-\Big(1-\frac{2+c}{t}\Big)(K_Y+{B_{\epsilon}})\sim_\bQ \Big(1-\frac{2+c}{t}\Big)\epsilon A
\end{align*}
is ample as $t>2(2+c)$. So by Lemma~\ref{CL},
$$\Nklt\Big(Y, \Big(1-\frac{2+c}{t}\Big){B_{\epsilon}}+\frac{2+c}{t}D+G_1+\pi^*C_0\Big)
$$
is connected. By construction, it
contains $F_z$ and prime divisors in $\Supp(\pi^*C_0)$ that dominate $C_0$. 
As $z\not \in C_0$, $F_z$ is disjoint from $\Supp(\pi^*C_0)$, so 
by \eqref{nklt Fz} and the connectedness, 
$\Nklt(Y, (1-\frac{2+c}{t}){B_{\epsilon}}+\frac{2+c}{t}D)$ is non-empty and intersects $F_z$.
As $z$ is general, this implies that
$\Nklt(Y, (1-\frac{2+c}{t}){B_{\epsilon}}+\frac{2+c}{t}D)$ dominates $S$.
 By applying the inversion of adjunction twice (see \cite[Lemma~5.50]{KM} or \cite[Theorem~2.11]{Jiang18}),
$(F, (1-\frac{2+c}{t}){B_{\epsilon}}|_{F}+\frac{2+c}{t}D|_{F})$ is not klt for a general fiber ${F}$  of $\pi$. 
As being klt is an open condition on coefficients, by the arbitrariness of $\epsilon$,
 it follows that $(F, (1-\frac{2+c}{t})\frac{1}{n}M|_{F}+\frac{2+c}{t}D|_{F})$ is not klt for a very general fiber ${F}$  of $\pi$. 

So there exists a closed point $P\in {F}$ such that 
\begin{align}
 \mult_P\Big(\Big(1-\frac{2+c}{t}\Big)\frac{1}{n}M|_{F}+\frac{2+c}{t}D|_{F}\Big)\geq 1.\label{mult p geq 1}
\end{align}
Note that ${F}\simeq \bP^1$. Hence $\mult_P (D|_{F})\leq\deg (D|_{F})=\deg(-K_{\bP^1})=2$. 
As $M$ is irreducible and $F$ is general, we may assume that $M\to S$ is \'etale over $\pi(F)$, and in particular, $\mult_P (M|_{F})\leq 1$. 
So \eqref{mult p geq 1} implies that
\begin{align*}
 \Big(1-\frac{2+c}{t}\Big)\frac{1}{n}+\frac{2(2+c)}{t}\geq 1.
\end{align*}
As we can choose $n$ arbitrarily large, the above inequality implies that $t\leq 2(2+c)$. But this contradicts the definition of $t$. So we conclude that $\Vol(Y, -K_{Y})\leq 24(2+c)d\leq 312$. 
\end{proof}

\subsection{Proofs of Theorem~\ref{mainthm} and Corollary~\ref{maincor}}
\begin{proof}[Proof of Theorem~\ref{mainthm}]
To prove Theorem~\ref{mainthm}, it suffices to show that for $Y$ in Proposition~\ref{prop MFS}, $\Vol(Y, -K_Y)\leq 324.$
If $Y$ satisfies (a) in Proposition~\ref{prop MFS}, then 
${\rm Vol}(-K_Y)=(-K_Y)^3\leq 64$ by \cite{Nami} and \cite[Theorem~1.2]{Prok07}.
If $Y$ satisfies (b) or (c) in Proposition~\ref{prop MFS}, then 
${\rm Vol}(-K_Y)\leq 324$ by Proposition~\ref{prop dp fibration} and Proposition~\ref{prop conic bdl}.
\end{proof}

\begin{proof}[Proof of Corollary~\ref{maincor}]
Let $X$ be a canonical weak $\bQ$-Fano $3$-fold. By \cite[Theorem~6.23, Theorem~6.25]{KM}, we can take a $\bQ$-factorial terminalization of $X$, that is, there is a projective birational morphism $\phi: W \to X$ such that $K_{W}=\phi^*(K_X)$ and $W$ is $\mathbb{Q}$-factorial terminal.
In particular, $(-K_X)^3=(-K_W)^3$ and $h^0(X, -K_X)=h^0(W, -K_W)$.

By Reid's Riemann--Roch formula, 
$$
h^0(W, -K_W)=\frac{1}{2}(-K_W)^3+\frac{1}{2}\sum_i \Big(\frac{b_i^2}{r_i}-b_i\Big)+3,
$$
where $\{(b_i, r_i)\mid 0<b_i\leq r_i/2\}$ runs over Reid’s basket of orbifold points of $W$ (see \cite[2.3]{Prok07} or \cite[2.2]{CJ16}). In particular, 
$
h^0(W, -K_W)\leq \frac{1}{2}(-K_W)^3+3,
$
and the equality holds if and only if $W$ is Gorenstein, or equivalently, $K_W$ is Cartier (cf. \cite[(2.6)]{Prok07}).

If $W$ is Gorenstein, then $(-K_W)^3\leq 72$ by \cite{Prok05} (see \cite[Theorem~1.5]{Prok05} and \cite[paragraph after Definition~1.10]{Prok05}), and hence $
h^0(W, -K_W)\leq 39$. If $W$ is not Gorenstein, then by Theorem~\ref{mainthm}, $
h^0(W, -K_W)< 165$.
\end{proof}

\section*{Funding} This work was supported by National Key Research and Development Program of China (Grant No.~2020YFA0713200) and National Natural Science Foundation of China for Innovative Research Groups (Grant No. 12121001).
\section*{Acknowledgments} The authors are grateful to Meng Chen, Jingjun Han, and Yujie Luo for discussions and suggestions. The second author would like to thank her advisor, Professor Meng Chen, for his support and encouragement. Part of this paper was written during the second author enjoying "The First National Number Theory and Algebraic Geometry Doctoral Forum" at the University of Science and Technology of China and she is grateful for the hospitality and support of USTC. We thank the referees for useful comments and suggestions.

\end{document}